\documentclass[a4paper,9pt]{article}
\usepackage{amsfonts}
\usepackage{amssymb,amsthm,color}
\usepackage{amsmath}

\newcommand{\ds}{\displaystyle}

\newcommand{\R}{\mathbb{R}}

\newcommand{\E}{\mathbb{E}}
\newcommand{\Pa}{\mathbb{P}}

\newtheorem{theorem}{Theorem}
\newtheorem{lemma}[theorem]{Lemma}
\newtheorem{proposition}[theorem]{Proposition}

\newtheorem{remark} {Remark}

\def\H{{\mathcal H}}

\newcommand{\Om}{\Omega}
\newcommand{\lb}{\lambda}
\newcommand{\Hm}{{\mathcal H}^{m-1}}
\newcommand{\sm}{\setminus}

\newcommand{\sq}{\subseteq}

\newcommand{\vphi}{\varphi}
\newcommand{\bp}{\begin{proof}}
\newcommand{\ep}{\end{proof}}

\begin{document}
\title{On the torsion function with Robin or Dirichlet boundary conditions}
\author{{M. van den Berg} \\
School of Mathematics, University of Bristol\\
University Walk, Bristol BS8 1TW\\
United Kingdom\\
\texttt{mamvdb@bristol.ac.uk}\\
\\
{D. Bucur}\\
Laboratoire de Math\'{e}matiques, CNRS UMR 5127\\
Universit\'{e} de Savoie Campus Scientifique\\
73376 Le-Bourget-du-Lac\\
France\\
\texttt{dorin.bucur@univ-savoie.fr}}
\date{26 April 2013}\maketitle
\vskip 3truecm \indent
\begin{abstract}\noindent
For $p\in (1,+\infty)$ and $b \in (0, +\infty]$ the $p$-torsion
function with Robin boundary
 conditions associated to an arbitrary open set $\Om \subset \R^m$ satisfies formally the equation
$-\Delta_p =1$ in $\Om$ and $|\nabla u|^{p-2} \frac{\partial u
}{\partial n} + b|u|^{p-2} u =0$ on $\partial \Om$. We obtain
bounds of the $L^\infty$ norm of $u$ {\it only} in terms of the
bottom of the spectrum (of the Robin $p$-Laplacian), $b$ and the
dimension of the space in the following two extremal cases: the
linear framework (corresponding to $p=2$) and arbitrary $b>0$, and
the non-linear framework (corresponding to arbitrary $p>1$) and
Dirichlet boundary conditions ($b=+\infty$). In the general case,
$p\not=2, p \in (1, +\infty)$ and $b>0$ our bounds involve also
the Lebesgue measure of $\Om$.

\end{abstract}
\begin{center} \textbf{Keywords}: Torsion function, Robin boundary conditions, $p$-Laplacian.
\end{center}
\mbox{}
\section{Introduction\label{sec1}}
Let $\Omega$ be an open set in Euclidean space $\R^m$ with
non-empty boundary $\partial \Omega$, and let the torsion function
$u:\Omega\rightarrow \R$ be the unique weak solution of
\begin{equation*}%\label{e1}
-\Delta u=1, \ \ \  u|_{\partial \Omega}=0.
\end{equation*}
If $\Omega$ has finite measure the solution is obtained in the
usual framework of the Lax-Milgram theorem, while if $\Omega$ has
infinite measure then $1 \notin H^{-1}(\Omega)$ and $u$ is defined
as the supremum over all balls of the torsion functions associated
to $\Omega \cap B$.

The torsional rigidity is the set function defined by
\begin{equation}\label{e2}
P(\Omega)=\int_{\Omega}u.
\end{equation}
Since $u\ge 0$ we have that $P(\Omega)$ takes values in the
non-negative extended real numbers, and that $P(\Omega)=\lVert
u\rVert_{L^1(\Omega)},$ whenever $u$ is integrable. Both the
torsion function and the torsional rigidity arise in many areas of
mathematics, for example in elasticity theory \cite{Bandle, KMN, PS}, in
heat conduction \cite{MvdB1}, in the definition of gamma
convergence \cite{BB}, in the study of minimal submanifolds
\cite{MP} etc. The connection with probability theory is as
follows. Let $(B(s), s\geq 0; \Pa_x, x \in \R^m)$ be Brownian
motion with generator $\Delta$, and let
\begin{equation*}%\label{e3}
T_{\Omega} = \inf \left\{ s\geq 0 \colon B(s) \in
\R^m\setminus \Omega \right\}
\end{equation*} be the first exit time of Brownian
motion from $\Omega$. Then \cite{PortStone}
\begin{equation*}%\label{e4}
u(x) = \E_x\left[ T_{\Omega}\right],\quad x \in \Omega.
\end{equation*}
Let $\lambda$ be the bottom of the spectrum of the Dirichlet Laplacian acting in $L^2(\Omega)$.

In \cite{MvdB2,MvdB3} it was shown that $u\in L^{\infty}(\Omega)$
if and only if $\lambda>0$. If $\lambda>0$ then
\begin{equation}\label{e5}
\lambda^{-1} \leq \Vert u
\Vert_{L^\infty(\Omega)} \leq \left( 4+3m\log 2\right) \lambda^{-1}.
\end{equation}
Previous results of this nature were obtained in Theorem 1 of
\cite{BC1} for open, simply connected, planar sets $\Omega$. The
question of the sharp constant in the upper bound in the right
hand side of \eqref{e5} for these sets was addressed in
\cite{BC1,BC2}.

In this paper we consider the torsion function $u_b$ for the
Laplacian with Robin boundary conditions. The Robin Laplacian is
generated by the quadratic form
\begin{equation*}%\label{e6}
\mathcal{Q}_b(u,v)= \int_{\Omega}\nabla u.\nabla v+b\int_{\partial
\Omega}uvd\mathcal{H}^{m-1},\ u,v \in W_{2,2}^1(\Omega, \partial \Omega),
\end{equation*}
where $\mathcal{H}^{m-1}$ denotes the $(m-1)$-dimensional
Hausdorff measure on $\partial \Omega$, and $b$ is a strictly
positive constant. This quadratic form defined on
$W_{2,2}^1(\Omega,\partial \Omega)$ is closed. The unique
self-adjoint operator generated by $\mathcal{Q}_b$ is the Robin
Laplacian which formally satisfies the boundary condition
\begin{equation}\label{e7}\frac{\partial u}{\partial
n}+bu=0,\ \  x\in \partial \Omega,\end{equation} where $n$ denotes
the outward unit normal, and $\frac{\partial}{\partial n}$ is the
normal derivative. The torsion function $u_b$ is the unique weak
solution of $-\Delta u=1$ with boundary condition \eqref{e7}. For
convenience we put $q_b(u)=\mathcal{Q}_b(u,u)$. It is well known
that $W_{2,2}^1(\Omega)=W^{1,2}(\Omega)$ if $\Omega$ is bounded
and $\partial \Omega$ is Lipschitz. See \cite{M} for details.
However, as all our results are for arbitrary open sets in $\R^m$
we will not rely on this identity. We denote the bottom of the
spectrum of the Robin Laplacian acting in $L^2(\Omega)$ by
\begin{equation}\label{e8}
\lambda(\Omega,b)=\inf\{q_b(u):\lVert u \rVert_{L^2(\Omega)}=1, u\
\textup{smooth in}\ \overline{\Omega}\}.
\end{equation}
The main results of this paper are the following.
\begin{theorem}\label{The1}
Let $\Omega$ be an open set in $\R^m, m=2,3,\cdots$. The torsion
function $u_b$ is bounded if and only if $\lambda(\Omega,b)>0$. In
that case we have that
\begin{equation}\label{e9}
\lambda(\Omega,b)^{-1}\le \lVert u_b\rVert_{L^{\infty}(\Omega)}\le
6m\lambda(\Omega,b)^{-1}\log\left(2^{11}3\sqrt3m(1+b^{-1}\lambda(\Omega,b)^{1/2})\right).
\end{equation}
\end{theorem}

For $b\rightarrow \infty$ we have that
$\lambda(\Omega,b)\rightarrow \lambda$, and we recover \eqref{e5}
with an albeit worse constant. For $b \rightarrow 0$ we have that
$\lambda(\Omega,b)\rightarrow 0$. However, in the case where
$\Omega$ is a $C^{0,1}$ domain it was shown in \cite{GS} that
\begin{equation}\label{e10}
\lim_{b\rightarrow
0}b^{-1}\lambda(\Omega,b)=|\Omega|^{-1}\H^{m-1}(\partial \Omega),
\end{equation}
where $|\Omega|=\int_{\Omega}1$. (The upper bound in \eqref{e10}
follows by choosing the test function $u=|\Omega|^{-1/2}$ in
\eqref{e8}.) So for these domains
$b^{-1}\lambda(\Omega,b)^{1/2}\asymp b^{-1/2},$ and the upper
bound in Theorem \ref{The1} has an extra factor $\log b$ compared
with the Dirichlet regime $b\rightarrow \infty$. It is unclear
whether this additional $\log b$ factor is in fact sharp.

The proof of Theorem \ref{The1} depends very heavily on the
availability of Gaussian upper bounds for the Robin heat kernel.
These were obtained in great generality in \cite{D2}. We note that
the estimates obtained in \cite{D1} and \cite{D3} for elliptic
Robin boundary value problems do not seem explicit enough to keep
track of the geometric data of $\Omega$. The remainder of this
paper is organised as follows. In Section \ref{sec2} we prove
Theorem \ref{The1}. In Section \ref{sec3} we obtain some bounds
for Robin eigenfunctions in the case where $\Omega$ has finite
measure.

In Section \ref{sec4} we study the torsion function and torsional
rigidity for the $p$-Laplacian with Dirichlet and Robin boundary
conditions respectively. In particular in Theorem \ref{bebu01} we
will obtain an $L^\infty$ estimate for the torsion function of the
$p$-Laplacian with Dirichlet boundary conditions for an arbitrary
open set in terms of the corresponding spectral bottom. This
extends the upper bound in \eqref{e5} for $p=2$ to all $p>1$.
Moreover it extends the results of \cite[Theorem 13]{BE} for
convex sets to arbitrary open sets with finite or infinite
measure. In the very general case of the $p$-Laplacian with Robin
boundary conditions, we obtain $L^\infty$ bounds which hold only
on open sets with finite measure, and the constant involve the
measure of $\Om$ as well. This last result is probably not
optimal, since one may expect that the Lebesgue measure should not
enter into the constant, but we are not able to overcome a series
of technical points.

\section{Proof of Theorem 1 \label{sec2}}
The main ingredient in the proof of Theorem \ref{The1} is a
Gaussian bound for the Robin heat kernel $p_b(x,y;t), x\in \Omega,
y\in \Omega, t>0$, for arbitrary open sets $\Omega$ in $\R^m, m=2,3,\cdots$. In
the special case where $\Delta$ is the standard Laplacian and
$b>0$ is constant on $\partial \Omega$, Theorem 6.1 in \cite{D2}
reads as follows.

For all $0<\epsilon\le 1$ and for all $x\in \Omega, y\in \Omega,
t>0$
\begin{equation*}%\label{e14}
p_b(x,y;t)\le C_{2m}(\alpha
\epsilon)^{-m}C^mt^{-m}e^{-|x-y|^2/(4\omega(1+\epsilon)t)},
\end{equation*}
where $\alpha=\min\{1,b\},$ and $C_{2m}$ and $\omega$ are constants
depending on $m$ only. $C$ is the constant which appears in the Nash
inequality
\begin{equation*}%\label{e15}
\lVert u \rVert_{L^2(\Omega)}^{2+\frac{2}{m}}\le Cq_1(u)\lVert u
\rVert_{L^1(\Omega)}^{\frac{2}{m}},
\end{equation*}
It is straightforward to trace the $m$-dependence of $\omega$. We
find that upon consulting Lemma 6.3 and its proof in \cite {D2},
\begin{equation*}%\label{e16}
\omega=1+m^{1/2}+4m.
\end{equation*}
Similarly we find that using Corollary 5.3 and Lemma 5.7 and their
proofs that
\begin{equation}\label{e17}
C_{2m}=(192m)^m.
\end{equation}
We can also verify that if $\mathcal{N}_b$ is a constant in the
Nash inequality
\begin{equation}\label{e18}
\lVert u \rVert_{L^2(\Omega)}^{2+\frac{2}{m}}\le
\mathcal{N}_bq_b(u)\lVert u \rVert_{L^1(\Omega)}^{\frac{2}{m}},
\end{equation}
and if we choose $\epsilon=1$ then we infer, by the previous lines,
that
\begin{equation}\label{e19}
p_b(x,y;t)\le C_{2m}\mathcal{N}_b^mt^{-m}e^{-|x-y|^2/(8\omega t)}.
\end{equation}
In Lemmas \ref{Lem3}, \ref{Lem4}, \ref{Lem5} and \ref{Lem6} below
we prove the Nash inequality \eqref{e18} with a constant
$\mathcal{N}_b$ depending upon $b, \lambda(\Omega,b)$ and $m$
only. As a first step we shall consider only open and bounded sets
$\Omega$ with a smooth boundary. By a standard density argument we
obtain the full $SBV$-case in a second step.
\begin{lemma}\label{Lem3}
There exists a constant $C(m)$ depending upon $m$ only such that
for all $v \in BV(\R^m)$
\begin{equation*}%\label{e20}
\lVert v \rVert_{L^{m/(m-1)}(\R^m)}\le C(m)|Dv|(\R^m),
\end{equation*}
where $|Dv|(\R^m)$ is the total variation of $v$ on $\R^m$, and
$C(m)$ is the isoperimetric constant given by
\begin{equation}\label{e21}
C(m)=m^{-1}\pi^{-1/2}(\Gamma((2+m)/2))^{1/m}.
\end{equation}
\end{lemma}
For a proof we refer to Theorem 3.4.7 in \cite{AFP}, or for an
elementary proof with (non-sharp) constant $1$ instead of $C(m)$
to Theorem 1 in Section 4.5.1 in \cite{EG}.
\begin{lemma}\label{Lem4}
Let $\Omega$ be an open bounded set in $\R^m$ and with smooth
boundary, and let $u\in H^1( \Omega)$. Then
\begin{equation*}%\label{e22}
\lVert u \rVert_{L^{2m/(m-1)}(\Omega)}^2\le
C(m)\left(2\int_{\Omega}|u||\nabla
u|+\int_{\partial\Omega}u^2d\H^{m-1}\right).
\end{equation*}
\end{lemma}
\begin{proof}
Let $u\in H^1( \Omega)$ and observe that $u^2\in BV(\R^m)$. In
fact we have that $u^2\in SBV(\R^m)$, where $u$ is extended by $0$
on $\R^m\setminus \Omega$. Indeed $u^2 \in L^1(\R^m)$ and for any
open set $A\subset \R^m$,
\begin{equation*}%\label{e23}
Du^2(A)=2\int_Au\nabla u+ \int_{\partial \Omega\cap A}
u^2\overrightarrow{n}d\H^{m-1}.
\end{equation*}
So
\begin{equation*}%\label{e24}
|Du^2|(\R^m)\le 2\int_{\Omega}|u||\nabla u|+ \int_{\partial
\Omega} u^2d\H^{m-1}.
\end{equation*}
By Lemma \ref{Lem3}
\begin{equation*}%\label{e25}
\lVert u^2 \rVert_{L^{m/(m-1)}(\Omega)}\le
C(m)\left(2\int_{\Omega}|u||\nabla u|+ \int_{\partial \Omega}
u^2d\H^{m-1}\right),
\end{equation*}
which implies the lemma.
\end{proof}
\begin{lemma}\label{Lem5}
For all $b>0$ and all $u \in H^1(\Omega)$,
\begin{equation}\label{e26}
\lVert u \rVert_{L^{2m/(m-1)}(\Omega)}^2\le
C(m)\left(\frac{1}{b}+\frac{b}{\lambda(\Omega,b)}\right)q_b(u),
\end{equation}
and for all $b\ge\lambda(\Omega,b)^{1/2}$ and $u \in H^1(\Omega),$
\begin{equation}\label{e27}
\lVert u
\rVert_{L^{2m/(m-1)}(\Omega)}^2\le2C(m)\lambda(\Omega,b)^{-1/2}q_b(u).
\end{equation}
\end{lemma}
\begin{proof}
In order to prove \eqref{e27} we use Cauchy-Schwarz and obtain
that
\begin{equation}\label{e28}
2\int_{\Omega}|u||\nabla u|\le b^{-1}\int_{\Omega}|\nabla
u|^2+b\int_{\Omega}u^2.
\end{equation}
So by Lemma \ref{Lem4}
\begin{align}\label{e29}
\lVert u \rVert_{L^{2m/(m-1)}(\Omega)}^2&\le
C(m)\left(b^{-1}\int_{\Omega}|\nabla
u|^2+b\int_{\Omega}u^2+\int_{\partial\Omega}u^2d\H^{m-1}\right)\nonumber \\
&\le C(m)b^{-1}\left(q_b(u)+b^2\int_{\Omega}u^2\right)\nonumber \\
& \le
C(m)b^{-1}\left(q_b(u)+b^2\lambda(\Omega,b)^{-1}q_b(u)\right),
\end{align}
which yields \eqref{e26}. In order to prove \eqref{e27} we replace
$b$ by $\lambda(\Omega,b)^{1/2}$ in \eqref{e28} and \eqref{e29}
respectively. This gives that
\begin{equation*}%\label{e30}
\lVert u
\rVert_{L^{2m/(m-1)}(\Omega)}^2\le2C(m)\lambda(\Omega,b)^{-1/2}q_{\lambda(\Omega,b)^{1/2}}(u)\le
2C(m)\lambda(\Omega,b)^{-1/2}q_b(u),
\end{equation*}
by monotonicity of $b\mapsto q_b$.
\end{proof}
Finally we obtain the following Nash inequality.
\begin{lemma}\label{Lem6}
For all $u\in H^1(\Omega)$ we have that
\begin{equation*}%\label{e31}
\lVert u \rVert_{L^2(\Omega)}^{2+\frac{2}{m}}\le \mathcal{N}_bq_b(u)\lVert u
\rVert_{L^1(\Omega)}^{\frac{2}{m}},
\end{equation*}
where
\begin{equation}\label{e32}
\mathcal{N}_b=\begin{cases}C(m)(b^{-1}+\lambda(\Omega,b)^{-1}b)&,\quad
b>0,\\
\hfill 2C(m)\lambda(\Omega,b)^{-1/2}&, \quad b\ge\lambda(\Omega,b)^{1/2}.\\
\end{cases}
\end{equation}
\end{lemma}
\begin{proof}
The lemma follows by Lemma \ref{Lem5} and the following interpolation
inequality
\begin{equation*}%\label{e33}
\lVert u \rVert_{L^2(\Omega)}\le \lVert u
\rVert_{L^{2m/(m-1)}(\Omega)}^{m/(m+1)}\lVert u
\rVert_{L^1(\Omega)}^{1/(m+1)}.
\end{equation*}
\end{proof}

For non-smooth $\Omega$ one can follow the same arguments if
instead of the Sobolev traces we consider $u^2\in SBV(\R^m)$, and
pointwise traces, or the Mazya trace (which has a priori higher
$L^2$ norm). See page 940 lines -10  to -1 and page 941 lines
1 to 6 in \cite{BG}.

In the lemma below we will use the heat equation techniques from \cite{MvdB2} that were used to obtain bounds
for the torsion function with Dirichlet boundary conditions. We abbreviate $K=C_{2m}\mathcal{N}_b^m$.
\begin{lemma}\label{Lem7}
Suppose that $\lambda(\Omega,b)>0$ and that \eqref{e19} holds. Let
$T>0$ be arbitrary. Then the torsion function with Robin boundary
conditions satisfies
\begin{equation*}%\label{e34}
\lVert u \rVert_{L^{\infty}(\Omega)}\le T+
2^{-1}(256\pi\omega)^{m/2}K\lambda(\Omega,b)^{-1}T^{-m/2}e^{-T\lambda(\Omega,b)/4}.
\end{equation*}
\end{lemma}
\begin{proof}
First note that Lemma 1 in \cite{MvdB2} holds for heat kernels
with Robin boundary conditions. By choosing $\beta=1/2$ in that
lemma we obtain that
\begin{equation*}%\label{e35}
p_b(x,x;t) \leq e^{-t \lambda(\Omega,b)/2}\, p_b\big(x,x;t/2).
\end{equation*}
Next note that by the heat semigroup property and the Cauchy-Schwarz inequality
\begin{align*}%\label{e36}
p_b(x,y;t) & =   \int_{\Omega} p_b(x,z;t/2)\, p_b(z,y;t/2)dz \nonumber \\  & \leq
\left( \int_{\Omega} p_b(x,z;t/2)^2dz\right)^{1/2} \left(\int_{\Omega}
p_b(z,y;t/2)^2\,dz\right)^{1/2}\nonumber \\  & = (p_b(x,x;t)\,
p_b(y,y;t))^{1/2}.
\end{align*}
So putting the above two estimates together with \eqref{e19} gives that
\begin{align*}%\label{e37}
p_b(x,y;t)&\le
(p_b(x,y;t))^{1/2}(p_b(x,x;t)p_b(y,y;t))^{1/4}\nonumber \\ & \le
K2^{m/2}e^{-t\lambda(\Omega,b)/4}t^{-m}e^{-|x-y|^2/(16\omega t)}.
\end{align*}
We obtain that, by extending the region of integration to all of $\R^m$,
\begin{equation*}%\label{e38}
\int_{\Omega}dyp_b(x,y;t)\le
K(32\pi\omega)^{m/2}t^{-m/2}e^{-t\lambda(\Omega,b)/4},
\end{equation*}
and
\begin{align}\label{e39}
\int_{[T,\infty)}dt\int_{\Omega}dyp_b(x,y;t)&\le
4(32\pi\omega)^{m/2}K\lambda(\Omega,b)^{-1}T^{-m/2}e^{-T\lambda(\Omega,b)/4}\nonumber
\\ & \le
2^{-1}(256\pi\omega)^{m/2}K\lambda(\Omega,b)^{-1}T^{-m/2}e^{-T\lambda(\Omega,b)/4}.
\end{align}
Furthermore
\begin{equation*}%\label{e40}
v(x;t)=\int_{\Omega}dyp_b(x,y;t)
\end{equation*}
is the solution of $\Delta v=\frac{\partial v}{\partial t}$ with initial condition $v(x;0)=1$ and Robin boundary conditions.
By the maximum principle we have that $0\le v(x;t) \le 1$. Hence
\begin{equation*}%\label{e41}
\int_{[0,T]}dt\int_{\Omega}dyp_b(x,y;t)\le T,
\end{equation*}
and the lemma follows by \eqref{e39} since the torsion function can be represented by
\begin{equation*}%\label{e42}
u(x)=\int_{[0,\infty)}dt\int_{\Omega}dyp_b(x,y;t).
\end{equation*}
\end{proof}

\noindent \emph{Proof of Theorem \ref{The1}}. We choose $T$ to be
the unique positive root of
\begin{equation}\label{e43}
(256\pi\omega)^{m/2}K\lambda(\Omega,b)^{-1}T^{-m/2}e^{-T\lambda(\Omega,b)/4}=T.
\end{equation}
We rewrite this, using the numerical value of $K$, as follows.
\begin{equation}\label{e44}
(T\lambda(\Omega,b))^{(2+m)/2}e^{T\lambda(\Omega,b)/4}=(2^{20}3^2m^2\mathcal{N}_b^2\lambda(\Omega,b)\pi\omega)^{m/2}.
\end{equation}
It is easily seen that $\mathcal{N}_b^2\lambda(\Omega,b)\ge 1$ for
all $b$ and all $\lambda(\Omega,b)$, and that the right hand side
of \eqref{e44} is at least $e^{1/4}$. We conclude that
$T\lambda(\Omega,b)\ge 1$. Hence
$e^{T\lambda(\Omega,b)/4}\le(2^{20}3^2m^2\mathcal{N}_b^2\lambda(\Omega,b)\pi\omega)^{m/2}$,
and
\begin{equation}\label{e45}
T\le 2m\lambda(\Omega,b)^{-1}
\log(2^{20}3^2m^2\mathcal{N}_b^2\lambda(\Omega,b)\pi\omega).
\end{equation}
By Lemma \ref{Lem7}, \eqref{e43} and \eqref{e45} we find that
\begin{equation}\label{e46}
\lVert u \rVert_{L^{\infty}(\Omega)}\le
3m\lambda(\Omega,b)^{-1}\log(2^{20}3^2m^2\mathcal{N}_b^2\lambda(\Omega,b)\pi\omega).
\end{equation}
To estimate the numerical constant under the $\log$ in the right
hand side of \eqref{e46} we first note that by \eqref{e32},
\begin{equation}\label{e47}
\mathcal{N}_b^2\lambda(\Omega,b)\le
4C^2(m)\left(1+b^{-1}\lambda(\Omega,b)^{1/2}\right)^2.
\end{equation}
By \eqref{e21}, \eqref{e46}, \eqref{e47} and the bounds $\omega\le
6m$ and $(\Gamma((2+m)/2))^{2/m}\le m/2$ we find that
\begin{equation*}%\label{e48}
\lVert u \rVert_{L^{\infty}(\Omega)}\le
3m\lambda(\Omega,b)^{-1}\log\left(2^{22}3^3m^2(1+b^{-1}\lambda(\Omega,b)^{1/2})^2\right).
\end{equation*}
This completes the proof of the right hand side in \eqref{e9}.

To prove the lower bound in \eqref{e9} we let {
$B(p,R)=\{x:|x-p|<R\}$, $\Omega_R=\Omega\cap B(p,R)$, and we
denote by $p_{b,R}(x,y;t)$ }the heat kernel with Robin boundary
conditions on $(\partial \Omega) \cap B(p,R)$ and Dirichlet
boundary conditions on $(\partial\Omega_R)\setminus ((\partial
\Omega) \cap B(p,R))$. The region $\Omega_R$ has finite volume and
the spectrum of the Laplacian with the corresponding mixed
boundary conditions is discrete. Denote the first eigenvalue by
$\tilde{\lambda}(\Omega_R,b)$ with corresponding eigenfunction
$\tilde{\phi}_{b,R}$. In Proposition \ref{Pro9} below we will see
that the first Robin eigenfunction on an open set with finite
Lebesgue measure is bounded. Following the proof of Proposition
\ref{Pro9} we will show that $\tilde{\phi}_{b,R}$ is also bounded.
We then have that
\begin{align}\label{e49}
u_{b,R}(x)&=\int_{[0,\infty)}dt\int_{\Omega}dyp_{b,R}(x,y;t)\nonumber \\
&\ge\int_{[0,\infty)}dt\int_{\Omega}dyp_{b,R}(x,y;t)\frac{\tilde{\phi}_{b,R}(y)}{\lVert\tilde{\phi}_{b,R}\rVert_{L^{\infty}(\Omega_R)}}\nonumber
\\ &
=\int_{[0,\infty)}dte^{-t\tilde{\lambda}(\Omega_R,b)}\frac{\tilde{\phi}_{b,R}(x)}{\lVert\tilde{\phi}_{b,R}\rVert_{L^{\infty}(\Omega_R)}}\nonumber
\\
&=\tilde{\lambda}(\Omega_R,b)^{-1}\frac{\tilde{\phi}_{b,R}(x)}{\lVert\tilde{\phi}_{b,R}\rVert_{L^{\infty}(\Omega_R)}}.
\end{align}
Taking first the supremum over all $x\in \Omega_R$ in the left
hand side of \eqref{e49}, and subsequently the supremum over all
$x\in \Omega_R$ in the right hand side of \eqref{e49} gives that
\begin{equation}\label{e50}
\lVert u_{b,R}\rVert_{L^{\infty}(\Omega_R)
}\ge\tilde{\lambda}(\Omega_R,b)^{-1}.
\end{equation}
Taking first the limit $R\rightarrow \infty$ followed by the same
limit in the right hand side of \eqref{e50} yields the lower bound
in \eqref{e9}. This completes the proof of Theorem \ref{The1}.

\section{Robin eigenfunctions \label{sec3}}
In this section we obtain some estimates for eigenfunctions of the
Robin Laplacian.
\begin{proposition}\label{Pro8}
Let $ \Omega$ be an open set in $\R^m$ with finite measure
$|\Omega|,$ and suppose that $\lambda(\Omega,b)>0$. Then the
spectrum of the Robin Laplacian acting in { $L^2(\Omega)$} is
discrete and for all $t>0$
\begin{equation}\label{e51}
\sum_{j=1}^{\infty}e^{-t\lambda_j(\Omega,b)}\le
C_{2m}\mathcal{N}^m_b|\Omega|t^{-m},
\end{equation}
where $C_{2m}$ and $\mathcal{N}_b$ are the constants in
\eqref{e17} and \eqref{e32} respectively and
$\{\lambda_j(\Omega,b):j\in \mathbb N\}$ are the eigenvalues of
the Robin Laplacian. Note that
$\lambda(\Omega,b)=\lambda_1(\Omega,b)$ in this case.
\end{proposition}
\begin{proof}
Integrating the diagonal element of the heat kernel over $\Omega$
 shows that the Robin heat semigroup is trace class.
Hence the Robin spectrum is discrete and this in turn implies
\eqref{e51}.
\end{proof}

It is well known that upper bounds on the heat kernel imply bounds for eigenfunctions. See
Example 2.1.8 in \cite{Davies}. The following below is another such instance.
\begin{proposition}\label{Pro9}
Let $ \Omega$ be an open set in $\R^m$ with finite measure
$|\Omega|,$ and suppose that $\lambda(\Omega,b)>0$. Let
$\{\phi_j:j\in \mathbb N\}$ denote an orthonormal set of
eigenfunctions corresponding to the eigenvalues in Proposition
\ref{Pro8}. Then $\phi_j\in L^1(\Omega)\cap L^{\infty}(\Omega)$
and for all $j\in \mathbb N$
\begin{equation}\label{e52}
\lVert
\phi_j\rVert_{L^{\infty}(\Omega)}\le(C_{2m}\mathcal{N}^m_be^mm^{-m})^{1/2}\lambda_j(\Omega,b)^{m/2}.
\end{equation}
\end{proposition}
\begin{proof}
By Cauchy-Schwarz and orthonormality we have that $\lVert
\phi_j\rVert_{L^{1}(\Omega)}\le |\Omega|^{1/2}.$ Since
$|\Omega|<\infty$ the heat semigroup is trace class, and so
\begin{equation*}%\label{e53}
e^{-t\lambda_j(\Omega,b)}\phi_j^2(x)\le\sum_{j=1}^{\infty}e^{-t\lambda_j(\Omega,b)}\phi_j^2(x)\le
C_{2m}\mathcal{N}^m_bt^{-m}.
\end{equation*}
Hence \begin{equation}\label{e54} |\phi_j(x)|^2\le
C_{2m}\mathcal{N}^m_be^{t\lambda_j(\Omega,b)}t^{-m}.
\end{equation}
Taking the supremum over all $x\in \Omega$ in the left hand side
of \eqref{e54} followed by taking the infimum over all $t>0$ in
the right hand side gives the bound in \eqref{e52}.
\end{proof}

To see that $\tilde{\phi}_{b,R},$ defined above \eqref{e49}, is
bounded we note that by \eqref{e9}, $p_{b,R}(x,y;t)\le
p_b(x,y;t)\le C_{2m}\mathcal{N}^m_bt^{-m}.$ So
$e^{-t\lambda(\Omega_R,b)}\phi^2_{b,R}(x)\le p_{b,R}(x,x;t)\le
C_{2m}\mathcal{N}^m_bt^{-m}$. This shows that $\phi_{b,R}$ is
bounded.

We note that the $L^{\infty}$ estimate in \eqref{e52} together
with \eqref{e49} implies the following comparison estimate between
torsion function and first eigenfunction.
\begin{equation}\label{e55}
u_b(x)\ge(C_{2m}\mathcal{N}^m_be^mm^{-m})^{-1/2}\lambda(\Omega,b)^{-1-\frac{m}{2}}\phi_1(x).
\end{equation}
For $b\ge\lambda(\Omega,b)^{1/2}$ we use the second inequality in
\eqref{e32} to obtain that
\begin{equation*}%\label{e56}
u_b(x)\ge C\lambda(\Omega,b)^{-1-\frac{m}{4}}\phi_1(x),
\end{equation*}
for some constant $C$ depending on $m$ only. This jibes with
Theorem 5.1 in \cite{MvdB3}. In general one cannot expect however,
that $u_b$ and $\phi_1$ are comparable. See Theorem 6 and the
discussion in Section 3 in \cite{MvdB3}. Similarly Theorem
\ref{The1} and Proposition \ref{Pro9} show that for
$b\ge\lambda(\Omega,b)^{1/2},$
\begin{equation*}%\label{e57}
\lVert u_b^{m/4}\phi_1\rVert_{L^{\infty}(\Omega)}\le C',
\end{equation*}
where $C'$ depends on $m$ only. This jibes with Theorem 5.2 in
\cite{MvdB3}, and completes the analogy with the Dirichlet case in
this regime.

\section{Torsion Function and torsional rigidity for the $p$-Laplacian  \label{sec4}}

\subsection{Dirichlet boundary conditions}

In this section we consider the $p$-Laplacian for $1<p<+\infty$
with Dirichlet boundary conditions, corresponding formally to $b
=+\infty$. Let $\Om$ be an open and bounded set of $\R^m$ and
$w_\Om$ (or simply $ w$) the torsion function of the $p$-Laplacian
with Dirichlet boundary conditions. It is the unique solution of
$$\min_{u \in W^{1,p}_0(\Om)} \frac1p\int_\Om |\nabla u|^p  -\int_\Om u .$$
Let
$$\lb:= \min\{ \frac{\int_\Om |\nabla u|^p }{\int_\Om | u|^p } : u \in W^{1,p}_0(\Om) \sm \{0\}\}$$
be the first eigenvalue of the { $p$-Laplacian} with Dirichlet
boundary conditions on $\Om$.

These notions extend to every open set, not necessarily bounded,
by replacing the torsion function and the first eigenvalue with
suitable definitions on unbounded sets. The first eigenvalue has
to be replaced by the spectral bottom, and the torsion function by
the following Borel function (possibly infinite valued) obtained
by inner approximations
 $$w_\Om(x)=\sup_{R>0} w_{\Om\cap {B(0,R)}} (x).$$

We shall prove the following.
\begin{theorem}\label{bebu01}
There exists a constant $C_{m,p}$ such that for any open set
$\Om\sq \R^m$

\begin{equation}\label{bebu06}
\|w\|_{L^{\infty}(\Om)} \le C_{m,p} \lb^{-\frac{1}{p-1}}.
\end{equation}
\end{theorem}
Theorem \ref{bebu01} extends the inequality obtained in
\cite{MvdB2} and \cite{MvdB3} to the $p$-Laplacian, for which we
give an elliptic proof. We refer the reader to \cite[Theorem
13]{BE}, where this inequality is proved for convex sets.

{\noindent \it Proof of Theorem \ref{bebu01}.} It is enough to
consider a smooth, bounded open set $\Om$ and approach a general
open set with an increasing sequence of smooth inner sets. We
extend the torsion function $w$ to all of $\R^m$ by $0$, and
denote the new function again by $w$. Then $w$ satisfies
$$-\Delta_p w \le 1,$$
in the sense that
\begin{equation}\label{ew}
\forall \vphi \in C_c^\infty (\R^m), \vphi \ge 0,\; \int_{\R^m}
|\nabla w |^{p-2}\nabla w \nabla \vphi  \le \int_{\R^m} \vphi .
\end{equation}

We prove the following Cacciopoli type inequality.
\begin{lemma}\label{bebu02}
For every $c_1 > 2^{p-1}$ there exists $c_2$ depending on $c_1$,
$m$ and $p$ such that for every $\theta \in W^{1, \infty} (\R^m)$
we have
$$\int_{{\R^m}} |\nabla (w\theta)|^p  \le c_1 \int_{{\R^m}} w |\theta|^p  + c_2 \int_{{\R^m}} |\nabla \theta |^p w^p.$$
\end{lemma}
\begin{proof}
Without loss of the generality, we may assume that $\theta \ge 0$.
By taking $\vphi=w\theta^p$ as a test function in \eqref{ew} it
suffices to prove that
$$\int_{{\R^m}} |\nabla (w\theta)|^p  \le c_1 \int_{{\R^m}} |\nabla w |^{p-2}\nabla w     \nabla (w |\theta|^p)
 + c_2 \int_{{\R^m}} |\nabla \theta |^p w^p .$$
Since
$$\int_{{\R^m}} |\nabla (w\theta)|^p  \le 2^{p-1} \int_{{\R^m}}\left( |\nabla w|^p \theta^p + |\nabla \theta|^p w^p\right) ,$$
and
$$\int_{{\R^m}} |\nabla w |^{p-2}\nabla w     \nabla (w |\theta|^p)   = \int_{{\R^m}} |\nabla w |^{p} \theta ^p  + p \int_{{\R^m}}  |\nabla w |^{p-2}\nabla w
    \nabla \theta \theta ^{p-1} w $$
it suffices to prove that \begin{align*}%\label{ew1}
 -p c_1\int_{{\R^m}} &|\nabla w |^{p-2}\nabla w     \nabla \theta\,
\theta ^{p-1} w\nonumber \\ &
  \le (c_1-2^{p-1})  \int_{{\R^m}} |\nabla w |^{p} \theta ^p + (c_2- 2^{p-1})  \int_{{\R^m}} |\nabla \theta|^p w^p {,}
  \end{align*}
or even
$$\int_{{\R^m}}  |\nabla w |^{p-1 }\theta ^{p-1} |\nabla \theta| w \le  \frac{c_1-2^{p-1}}{pc_1}  \int_{{\R^m}} |\nabla w |^{p} \theta ^p +
  \frac{c_2- 2^{p-1}}{pc_1}  \int_{{\R^m}} |\nabla \theta|^p w^p.$$
This last inequality is a consequence of Young's inequality, for
$c_2$ given by
$$c_2=2^{p-1}+ c_1 \Big ( \frac{(p-1)c_1}{c_1-2^{p-1}}\Big )^{p-1}.$$
\end{proof}
In order to get a pointwise bound of $w$ in terms of the average
of $w$ on balls, we recall the following result from \cite{ma96}
(see also \cite{tr67}).

\begin{lemma}\label{bebu03}

Let $p \in (1,m]$ and $u \in W^{1,p}(\R^m)$, $u \ge 0$ and
$-\Delta_p u \le 1$. Let $\gamma \in (p-1,
\frac{N(p-1)}{N-(p-1)})$. Then there exist two constants $C, C'$
independent { of} $u$ such that
$$u(0) \le C\Big ( \int_{{ B(0,1)}}u ^\gamma \Big )^\frac{1}{\gamma} +
C'.$$
\end{lemma}
\begin{proof}
This is a consequence of \cite[Theorem 3.3]{ma96}.
\end{proof}

\noindent We now continue our proof of Theorem \ref{bebu01}.
Clearly, by re-scaling we get that
\begin{equation*}%\label{bebu04}
u(0) \le C\Big ( \frac{1}{r^m}\int_{{ B(0,r)}}u ^\gamma  \Big
)^\frac{1}{\gamma} + C' r^\frac{p}{p-1}.
\end{equation*}
We shall { choose} $\gamma \in (p-1,p)$, close to $p-1$. For
such a $\gamma$ we have { by} H\"older's inequality
$$\Big ( \frac{1}{r^m}\int_{{ B(0,r)}}u ^\gamma \Big )^\frac{1}{\gamma}
\le \Big ( \frac{1}{r^m}\int_{{ B(0,r)}}u ^p  \Big )^\frac{1}{p}\omega_m^{1-\frac{\gamma}{p}}.$$
So changing the constant $C$ we have for $p \in (1,m]$
\begin{equation}\label{bebu05}
u(0) \le C\Big ( \frac{1}{r^m}\int_{{ B(0,r)}}u ^p  \Big
)^\frac{1}{p} + C' r^\frac{p}{p-1}.
\end{equation}
If $p>m$, this inequality holds as well. This is a consequence of
the continuous embedding of $W^{1,p} (B_1(0))$ in $L^\infty
(B_1(0))$ and of Lemma \ref{bebu02}. Indeed, from Lemma \ref
{bebu02}, there exist constants (which may change from line to
line) such that
$$\int_{{ B(0,1/2)}} |\nabla u|^p  \le c_1 \int_{{ B(0,1)}} u + c_2\int_{{ B(0,1)}} u^p $$
$$\le C  \int_{{ B(0,1)}} u^p  + C'.$$
On the other hand
$$\|u\|_{L^\infty ({ B(0,1/2)})} \le C \|u\|_{L^p ({ B(0,1/2)})} + C\|\nabla u\|_{L^p ({B(0,1/2)})}, $$
so { that}
$$\|u\|_{L^\infty ({ B(0,1/2)})} \le   C \|u\|_{L^p ({ B(0,1)})} + C'.$$
By re-scaling, we obtain inequality \eqref{bebu05}.

Let now $\theta \in C_c^\infty (B_2(0)$, $0\le \theta \le 1$,
$\theta \equiv 1$ on $B_1(0)$. Let $\theta_R(x)= \theta
(\frac{x}{R})$. Then $w\theta \in W^{1,p}_0(\Om)$, so
$$\lb(\Om) \le \frac{\ds \int_\Om |\nabla (w\theta)|^p }{\ds \int_\Om w^p\theta^p }$$
and from Lemma \ref{bebu02}
$$\lb(\Om)  \ds \le \frac{\ds c_1 \int_\Om  w \theta^p  + c_2 \int_\Om |\nabla \theta |^p w^p }{\ds \int_\Om w^p\theta^p }.$$
Using inequality \eqref{bebu05} we have {for $R$} small
enough
$$\frac{R^m}{C^p} \Big ( w(0)-C'R^\frac{p}{p-1}\Big )^p \le \int_{{ B(0,R)}}w^p.$$
{ At} the same time
$$\int_{{ B(0,2R)}}\theta_R^p  \le \omega_m 2^m R^m\;\;\mbox{and}\;\; \int_{{B(0,2R)}}|\nabla \theta_R|^p \le C''R^{m-p}.$$
As a consequence, renaming constants, we get
$$\lb(\Om)  \ds \le \frac{\ds C w(0)R^m + C'w(0)^pR^{m-p}}{R^m \Big ( w(0)-C''R^\frac{p}{p-1}\Big )^p}.$$
We choose $R$ such that
$$w(0)= 2C''R^\frac{p}{p-1},$$
and so
$$\lb(\Om) \le \frac{\tilde C}{R^p}
=\frac{2^{p-1}C''^{p-1} \tilde C}{w(0)^{p-1}},$$ where
$$\tilde C= \frac{2CC''+ (2C'')^p C'}{C''^p}.$$
This concludes the proof of inequality
\eqref{bebu06}.\hspace*{\fill }$\square $

%\end{proof}

\subsection{Robin boundary conditions}

The $p$-torsion function with Dirichlet boundary conditions on the
ball $B(0,R)$ is given by
\begin{equation*}%\label{e571}
u(x)=(p-1)p^{-1}m^{-(p-1)^{-1}}(R^{p/(p-1)}-|x|^{p/(p-1)}), \
|x|<R.
\end{equation*}
Hence the solution is bounded, positive and regular.  It follows
by the comparison- and regularity theorems that if $\Omega\subset
B(0,R)$ then the $p$-torsion function for $\Omega$ with Dirichlet
boundary conditions satisfies
\begin{equation*}%\label{e572}
\lVert u \rVert_{L^{\infty}(\Omega)}\le
(p-1)p^{-1}m^{-(p-1)^{-1}}R^{p/(p-1)}.
\end{equation*}
See \cite{BE} and the references therein. In this section we
obtain some results for the $p$-torsion function with Robin
boundary conditions and the corresponding torsional rigidity.

Let $p>1$ and $\Om \subset \R^m$ be an open set of finite measure.
We introduce the torsion function for the $p$-Laplacian with Robin
boundary conditions relaying on the $W^1_{p,p}(\Om,
\partial \Om)$-spaces (see \cite{M}). All results of this section can be
rephrased in the framework introduced in \cite{BG}, where the
Robin problem in non-smooth sets is defined by using the $SBV$-
spaces (see Remark \ref{0989.5} at the end of the section).

The torsion function $u_b$ is the unique weak solution of
\begin{equation*}%\label{e58}
-\Delta_p u=1  \mbox { in } { \Omega,}\ \ \  |\nabla u|^{p-2} \frac{\partial u}{\partial n} + b |u|^{p-2} u=0 \mbox { on } {\partial \Omega},
\end{equation*}
which is the minimizer in $W^1_{p,p}(\Om, \partial \Om)$ of the energy
\begin{equation}\label{e59}
v \mapsto \int_\Om |\nabla v|^p  +b \int_{\partial \Om} |v|^p d\Hm
-p \int_\Om v .
\end{equation}

Let us notice that $u$ is non-negative and continuous in $\Om$. We
introduce the family of open sets $U_t=\{u > t, t\ge 0\}$ and
denote by $\lb (\Omega, b)$ the first Robin eigenvalue for the
open set $\Omega$ associated to the Robin constant $b$, which is
defined as
\begin{equation}\label{e60}\lb (\Omega, b) = \inf \Big \{ \frac{ \int_\Om |\nabla v|^p  +b \int_{\partial \Om}
 |v|^p d\Hm}{  \int_\Om | v|^p } : v \in W^1_{p,p}(\Om, \partial \Om), v\not=0\Big \}.
 \end{equation}
 Throughout this section we suppress the $p$-dependence of the first
 Robin eigenvalue, torsion function etc.
The isoperimetric inequality for the first
eigenvalue of the Robin $p$-Laplacian in a non-smooth setting (see
\cite{BG}) states that
\begin{equation}\label{e61}
\lb (U_t, b) \ge \lb (U^*_t, b),
\end{equation}
where we adopt the usual notation: for any measurable set $A$ with finite Lebesgue measure $A^*$ stands for the ball of measure ${|A|}$ centered at $0$.

Let us define the constants
\begin{equation*}%\label{e62}
 c_1=
\frac{m}{m(p-1)+1},\end{equation*}
\begin{equation*}%\label{e63}
c_2=p^{\frac{m}{m(p-1)+1}}(m(p-1)+1),
\end{equation*}
\begin{equation*}%\label{e64}
c_3=\frac{1}{(p-1)(m(p-1)+1)}, \end{equation*}
\begin{equation*}\label{e65} h(t)= \frac{|\Om|^{\frac{1}{m}} }
{|U_t|^{\frac{1}{m}}}. \end{equation*}
 We put $U_t^\sharp =h(t)
U_t$ so that $|U_t^\sharp|=|\Omega|.$ We prove the following result.
\begin{theorem}\label{The3}
For every open set of finite measure, the torsion function $u_b$ belongs to $L^{\infty} (\Om)$ and
\begin{equation}\label{e66}
\int_0^ {\|u_b\|_{L^{\infty}(\Omega)}}   \left(h(t)^{p-1} \lb
(\Omega ^*, \frac{b}{h(t)^{p-1}})\right)^ {c_1}  dt \le\frac {
c_2} {\lb (\Om, b)^{c_3}}.
\end{equation}
\end{theorem}

In the proof of Theorem \ref{The3} we will need the following.
\begin{lemma}\label{Lem8} Let $B$ be any open ball in $\R^m$ with finite Lebesgue measure. Then
\begin{equation*}%\label{e67}
\lim_{\alpha  \downarrow 0 } \frac{ \lb (B, b \alpha )}{\alpha} = mb .
\end{equation*}
\end{lemma}
\begin{proof}
{Let $u$ be the first normalized eigenfunction on $B(0,1)$}.
 The mapping $[0,1]\ni r \mapsto b(r)= \frac{|\nabla u(r)|^{p-1}}{u(r)^{p-1}}$ is increasing and continuous (see for instance \cite[Proposition 4.2]{BD}).
  Moreover, we have
\begin{equation*}%\label{e68}
{ \lb (B(0,r), b(r))=\lb(B(0,1), b).}
\end{equation*}
By { re-scaling} we get that
\begin{equation*}%\label{e69}
\frac{1}{r^p} \lb { (B(0,1), b(r) r^{p-1}) = \lb(B(0,1), b),}
\end{equation*}
so that
\begin{equation*}%\label{e70}
\frac{1}{b(r)r^{p-1}} \lb ({ B(0,1)}, b(r) r^{p-1})
=\frac{r}{b(r)} \lb({ B(0,1)}, b).
\end{equation*}
It remains to prove that
\begin{equation*}%\label{e71}
\lim_{r \downarrow 0} \frac{r}{b(r)} \lb({ B(0,1)}, b)=
\frac{\Hm(\partial { B(0,1)})}{|{ B(0,1)}|}.
\end{equation*}
Multiplying the identity $-\Delta_p u =\lb({ B(0,1)}, b)
|u|^{p-2}u$ with $u$ and integrating on ${ B(0,r)}$ we get
\begin{equation*}%\label{e72}
\int_{{B(0,r)}} |\nabla u|^p  + b(r) \int_{\partial {
B(0,r)}} |u|^pd \Hm = \lb ({ B(0,1)},b)\int_{{
B(0,r)}}|u|^p.
\end{equation*}
Dividing by $r^m$ and passing to the limit $r \downarrow 0$, we get
\begin{equation*}%\label{e73}
 \lim_{r \downarrow 0}\frac{b(r)}{r} \Hm(\partial { B(0,1)}) |u(0)|^p=  \lb ({ B(0,1)},b) |u(0)|^p |{ B(0,1)}|,
\end{equation*}
where we have used that $\nabla u (0)=0$ and $u(0)\not= 0$ by the regularity of $u$ inside the ball.
\end{proof}

Lemma \ref{Lem8} above extends \eqref{e9} to the $p$-Laplacian
with Robin boundary conditions in the special case of a ball.

\noindent\textit{Proof of Theorem \ref{The3}.}  As usual in the
search of $L^\infty$ estimates, we choose $u\wedge t:=\min
\{u,t\}$ as a test function in \eqref{e59}. We have that
\begin{align*}%\label{e75}
&\int_ {\Omega}|\nabla u|^p  + b \int_{\partial \Omega} |u|^p d
\Hm -p \int_\Omega u   \nonumber \\ &  \le  \int_ {\Omega}|\nabla
(u\wedge t)|^p  + b \int_{\partial \Omega} |u\wedge t|^p d \Hm -p
\int_\Omega (u\wedge t).
\end{align*}
But
\begin{equation*}%\label{e76}
\int_ {\Omega}|\nabla (u\wedge t)|^p  = \int_{\{0<u<t\}}|\nabla u|^p = \int_ {\Omega}|\nabla u|^p -\int_ {U_t}|\nabla u|^p ,
\end{equation*}
\begin{equation*}%\label{e77}
\int_\Omega (u\wedge t) =\int_ {\Omega\sm U_t}u  + \int_ { U_t}t
=\int_ {\Omega}u  - \int_ {\Omega}(u-t)^+ ,
\end{equation*}
and
\begin{align*}%\label{e78}
 \int _{\partial \Omega} |u\wedge t|^p d \Hm&= \int_{\partial \Omega \cap \{0<u<t\}} |u|^p d \Hm+ \int _{\partial \Omega \cap \{u\ge t\}} t^p d \Hm \nonumber \\ &
=\int _{\partial \Omega} |u|^p d \Hm-  \int _{\partial \Omega \cap \{u\ge t\}} (|u|^p-t^p) d \Hm \nonumber \\ &
=\int _{\partial \Omega} |u|^p d \Hm- \int_{\partial U_t}  (|u|^p-t^p) d \Hm.
 \end{align*}
This last equality is a consequence of the fact that for every $x \in \partial U_t \cap \Omega$, we have (from the continuity of $u$) that $u(x)=t$.
Finally,
\begin{equation*}%\label{e79}
\int_{U_t} |\nabla (u-t)^+|^p  +b \int_{\partial U_t} (|(u-t)^+
+t|^p-t^p) d\Hm\le p \int_{U_t} (u-t)^+ ,
\end{equation*}
so that by the definition of the first Robin eigenvalue
\begin{align*}%\label{e80}
\lb(U_t, b) \int_{U_t} | (u-t)^+|^p  &\le
\int_{U_t} |\nabla (u-t)^+|^p  +b \int_{\partial U_t} |(u-t)^+|^p d\Hm\nonumber \\ &\le p \int_{U_t} (u-t)^+ .
\end{align*}
Let $f(t)=\int_{U_t}(u-t)^+, t\ge 0 $, then H\"older's inequality
gives that
\begin{equation*}%\label{e81}
\lb(U_t, b) f(t)^p\frac{1}{|U_t|^\frac{p}{p'}} \le p f(t).
\end{equation*}
Equivalently, introducing the { re-scaling} of $U_t$, we get
that
\begin{equation*}%\label{e82}
f(t)^{p-1} h(t)^p \lb(U_t^\sharp, \frac{b}{h(t)^{p-1}}) \le p |U_t|^{p-1},
\end{equation*}
or
\begin{equation*}%\label{e83}
f(t)^{p-1} h(t)^{p-1} \lb(U_t^\sharp, \frac{b}{h(t)^{p-1}})|\Om|^{\frac{1}{m}} \le p |U_t|^{p-1+\frac{1}{m}}.
\end{equation*}
Since $f'(t)= -|U_t|$, we get that
\begin{equation*}%\label{e84}
\frac{|\Om|^{\frac{1}{m(p-1)+1}}} {p^{\frac{m}{m(p-1)+1}}} \big
[h(t)^{p-1} \lb(U_t^\sharp, \frac{b}{h(t)^{p-1}})\big ]
^{\frac{m}{m(p-1)+1}}
\le-\frac{f'(t)}{f(t)^{\frac{m(p-1)}{m(p-1)+1}}}.
\end{equation*}
Integrating this differential inequality between $0$ and some
value $T <\lVert u\rVert_{L^{\infty}(\Omega)}$, we get that
\begin{align*}%\label{e85}
 \frac{|\Om|^{\frac{1}{m(p-1)+1}}} {p^{\frac{m}{m(p-1)+1}}}&  \int _0^T  \left(h(t)^{p-1} \lb(U_t^\sharp, \frac{b}{h(t)^{p-1}})\right) ^{\frac{m}{m(p-1)+1}} dt \nonumber \\ & \le
(m(p-1)+1) (f(0)^{\frac{1}{m(p-1)+1}} -f(T)^{\frac{1}{m(p-1)+1}}).
\end{align*}
By the isoperimetric inequality \eqref{e61} and the positivity of $f$ we obtain that
\begin{align}\label{e86}
\frac{|\Om|^{\frac{1}{m(p-1)+1}}} {p^{\frac{m}{m(p-1)+1}}} & \int _0^T  \left(h(t)^{p-1} \lb(\Omega^*, \frac{b}{h(t)^{p-1}})\right) ^{\frac{m}{m(p-1)+1}} dt \nonumber \\ & \le
(m(p-1)+1) f(0)^{\frac{1}{m(p-1)+1}}.
\end{align}
Since
\begin{align}\label{e87}
\|u_b\|_{L^1(\Omega)}& =  \int_\Om |\nabla u_b|^p  +b \int_{\partial
\Om} |u_b|^p d\Hm \nonumber \\ & \ge \lb (\Om, b) \|u_b\|_{L^p(\Omega)}^p \nonumber \\ & \ge  \lb
(\Om, b) \frac{\|u_b\|_{L^1(\Omega)}^p}{|\Om|^{p-1}},
\end{align}
we have that
\begin{equation}\label{e88}
\|u_b\|_{L^1(\Omega)} \le  \frac{|\Om|}{\lb (\Om,
b)^{\frac{1}{p-1}}}.
\end{equation}
By \eqref{e86}, \eqref{e87}, \eqref{e88} and the fact that
$f(0)=\lVert u \rVert_{L^1(\Omega)}$ we conclude that
\begin{equation}\label{e89}
\frac{|\Om|^{\frac{1}{m(p-1)+1}}} {p^{\frac{m}{m(p-1)+1}}}  \int _0^T  \left(h(t)^{p-1} \lb(\Omega^*, \frac{b}{h(t)^{p-1}})\right) ^{\frac{m}{m(p-1)+1}} dt \le \frac{c_2}{\lb (\Om, b)^{c_3}}.
\end{equation}
Since $h(t) \rightarrow +\infty $ as $t \uparrow
\|u\|_{L^{\infty}(\Omega)}$, we have by Lemma \ref{Lem8} that
\begin{equation*}%\label{e90}
\lim_{t \uparrow \|u\|_{\infty }}h(t)^{p-1} \lb (\Omega ^*, \frac{b}{h(t)^{p-1}})=mb.
\end{equation*}
Hence $\|u_b\|_{L^{\infty}(\Omega) }<\infty$. Then choosing $T=
\|u_b\|_{L^{\infty}(\Omega) }$ in \eqref{e89} completes the proof
of Theorem \ref{The3}.\hspace*{\fill }$\square $

In analogy with \eqref{e2} we define the torsional rigidity for
the $p$-Robin torsion function by
\begin{equation*}%\label{e11}
P(\Omega,b)=\int_{\Omega}u_b,
\end{equation*}
where $u_b$ is a minimiser of \eqref{e59}. It is easily seen that $u_b\ge 0$. Hence $P(\Omega,b)=\lVert u \rVert_{L^1(\Omega)}.$
\begin{theorem}\label{The2}
If $\Omega$ is an open set in $\R^m, m=2,3,\cdots$ with $|\Omega|<\infty,$ and if $p>1$ and
$b>0$ then
\begin{equation*}%\label{e12}
b^{-1/(p-1)}|\Omega|^{p/(p-1)}\Hm(\partial \Omega)^{-1/(p-1)}\le
P(\Omega,b)\le \lambda(\Omega,b)^{-1/(p-1)}|\Omega|.
\end{equation*}
\end{theorem}
\begin{proof}
The infimum in \eqref{e59} is attained by the $p$-torsion function $u_b\in W_{p,p}^1(\Omega)$. We have the following
variational characterization.
\begin{equation*}%\label{e91}
P(\Omega,b)^{p-1}=\sup\left(\int_{\Omega} |\nabla v|^p +b
\int_{\partial\Omega} |v|^p
d\Hm\right)^{-1}\left\lvert\int_{\Omega} v \right\rvert^p,
\end{equation*}
where the supremum is over all $v\in
W_{p,p}^1(\Omega)\setminus\{0\}$. To prove the lower bound we
choose the test function $v=1$. To prove the upper bound we have
by H\"older's inequality that $\left(\int_{\Omega} v\right)^p\le
|\Omega|^{p-1}\int_{\Omega}|v|^p$. The variational
characterization of  $\lambda(\Omega,b)$ in \eqref{e60} gives that
$\int_\Om |\nabla v|^p  +b \int_{\partial \Om} |v|^pd\Hm\ge
\lambda(\Omega,b)\int_{\Omega}|v|^p$, which completes the proof.
\end{proof}

\begin{remark}\label{0989.5}
{\rm If $\Om$ is an open set with non-smooth boundary then the
space $W^1_{p,p}(\Om, \partial \Om)$ does not lead to the natural
relaxation of the Robin problem. Precisely if $\Om$ is a Lipschitz
set from which one removes a Lipschitz crack, in the space
$W^1_{p,p}(\Om, \partial \Om)$ all functions have the same trace
on both sides of the crack, while one can write properly the Robin
problem in the Sobolev space $W^{1,p}(\Om)$ which is much larger
than $W^1_{p,p}(\Om, \partial \Om)$.  A suitable relaxation of the
Robin problem, based on the special functions with bounded
variations, was introduced in \cite{BG} to deal with these
situations (see \cite{AFP} for details). The torsional rigidity
could be defined on the open bounded sets in $\R^m$ by $
P(\Omega)=\int_{\Om}u dx$,  where $u$ is any minimizer of
\begin{equation*}%\label{e92}
 v \mapsto \int_\Om |\nabla v|^p  +b \int_{J_v}  (|v^+|^p+|v^-|^p) d\Hm -p \int_\Om v ,
\end{equation*}
among all  non-negative  functions $v \in L^p(\R^m)$ such
that $v^p \in SBV(\R^m)$, $v=0$ a.e. on $\R^m \setminus \Om$,
$\Hm(J_v\sm
\partial \Om)=0$. Above $J_v$ is the set of jump points of $v$ and
$v^+$ and $v^-$ stand for the upper and lower approximate limits
of $v$ in a point of the jump set. A similar estimate as the one in \eqref{e66} can be obtained in the SBV framework, by replacing the $W^1_{p,p}$-Robin
eigenvalues with the SBV-Robin eigenvalues.

}
\end{remark}

\end{document}